\documentclass{amsart}
\usepackage{amsmath,amssymb,hyperref,mathrsfs,graphicx,bbm}
\usepackage[utf8x]{inputenc}
\usepackage{amsmath}
\usepackage{amsfonts}
\usepackage{latexsym}
\usepackage{amsthm}
\usepackage{stmaryrd,amscd}
\usepackage{xargs}
\usepackage{tikz}
\usepackage{graphicx}
\usepackage{color}
\usepackage{enumitem}
\usepackage{graphicx}
\usepackage{color}

\usepackage[colorinlistoftodos,prependcaption]{todonotes}
\presetkeys%
{todonotes}%
{inline,backgroundcolor=white}{}
\tikzset{/tikz/notestyleraw/.append style={text=black}}

\newtheorem{thm}{Theorem}

\newtheorem{lem}{Lemma}[section]
\newtheorem{defn}[lem]{Definition}
\newtheorem{prop}[lem]{Proposition}

\newtheorem{que}{Open Problem}

\newtheorem{rmk}[lem]{Remark}
\newcommand{\be}{\begin{eqnarray}}
\newcommand{\ee}{\end{eqnarray}}
\newcommand{\beq}{\begin{equation}}
\newcommand{\eeq}{\end{equation}}
\newcommand{\ben}{\begin{eqnarray*}}
\newcommand{\een}{\end{eqnarray*}}
\newcommand{\beal}{\begin{aligned}}
\newcommand{\enal}{\end{aligned}}

\newcommand{\eps}{\epsilon}

\newcommand{\gm}{\gamma}

\newcommand{\lb}{\lambda}

\newcommand{\T}{\mathbb{T}}
\newcommand{\N}{\mathbb{N}}

\newcommand{\R}{\mathbb{R}}

\newcommand{\Z}{\mathbb{Z}}

\newcommand{\om}{\omega}
\newcommand{\Om}{\Omega}
\newcommand{\Dt}{\Delta}

\newcommand{\cU}{\mathcal{U}}
\newcommand{\cA}{\mathcal{A}}
\newcommand{\cL}{\mathcal{L}}

\newcommand{\cT}{\mathcal{T}}

\newcommand{\wt}{\widetilde }

\title[Variational attraction of the KAM torus for CS-systems]{Variational attraction of the KAM torus for conformally symplectic systems}
\thanks{Email: $\dagger$ jl@njust.edu.cn,\quad $\ddagger$ jellychung1987@gmail.com,\quad $*$zhao\_kai@fudan.edu.en}
\date{}
\begin{document}
\maketitle

\centerline{\scshape Liang Jin$^\dagger$}

{\footnotesize
\centerline{Department of Applied Mathematics, School of Science}
 \centerline{Nanjing University of Sciences and Technology, Nanjing 210094, China}
}
\bigskip

\centerline{\scshape Jianlu Zhang$^\ddagger$}

{\footnotesize
\centerline{Hua Loo-Keng Key Laboratory of Mathematics \&}
 \centerline{Mathematics Institute, Academy of Mathematics and systems science}
 \centerline{Chinese Academy of Sciences, Beijing 100190, China}
}
\bigskip

\centerline{\scshape Kai Zhao$^*$}

{\footnotesize
\centerline{School of Mathematical Sciences}
\centerline{Fudan University, Shanghai 200433, China}
}
\bigskip

\begin{abstract}
For the conformally symplectic system
\[
\left\{
\begin{aligned}
\dot{x}&=H_p(x,p),\quad(x,p)\in T^*\T^n\\
\dot p&=-H_x(x,p)-\lb p, \quad\lb>0
\end{aligned}
\right.
\]
with a positive definite Hamiltonian, 
we discuss the variational significance of invariant Lagrangian graphs and explain
how the presence of the KAM torus impacts the $W^{1,\infty}-$ convergence speed of the Lax-Oleinik semigroup.
\end{abstract}

\vspace{10pt}


\section{Introduction}\label{s1}
\vspace{10pt}

The earliest research on {\sf conformally symplectic systems} can be found in Duffing's experimental designing book published in 1918, concerning forced oscillations with variable frequency \cite{D}. His work  inspires the creation of modern qualitative theory for dynamical systems, and a bunch of interesting phenomena, e.g. chaos, bifurcation, resonance etc \cite{GH} were found since then, although contemporaries of Duffing have  also noticed these objects, e.g. Poincar\'e \cite{P} and Lyapunov. 
Such a kind of systems have a wide practical prospect, which  can be found in almost all the modern scientific subjects,  e.g. astronomy \cite{CC}, electromagnetics \cite{M}, elastomechanics \cite{MH}, and even economics \cite{L}. The study of invariant Lagrangian submanifolds for dissipative systems, and in particular the existence of KAM tori (i.e., invariant Lagrangian tori on which the motion is conjugate to a rotation), have been deeply investigated in \cite{CCD1,CCD, Ma}. Besides, the PDE viewpoint and variational method provide more viewpoints towards the global dynamics \cite{DFIZ,L,MS}.

\subsection{Conformally symplectic system: Hamiltonian/Lagrangian formalism} For a $C^r$ smooth Hamiltonian function ($r\geqslant 2$) 
\be
H:(x,p)\in T^*\T^n\rightarrow\R
\ee
which satisfies
\begin{enumerate}
  \item[\bf (H1)] {\sf (Positive Definite)} $H_{pp}$ is positive definite everywhere on $T_x^*\T^n$;
  \item[\bf (H2)] {\sf (Superlinear)} $\lim_{|p|_{x}\rightarrow+\infty}H(x,p)/|p|_x=+\infty$, where $|\cdot|_x$ is the Euclidean norm on $T^{*}_{x}\T^n$;
\end{enumerate}
 we introduce a dissipative equation by
\be\label{eq:ode0}
\left\{
\begin{aligned}
\dot{x}&=H_p(x,p),\\
\dot p&=-H_x(x,p)-\lb p.
\end{aligned}
\right.
\ee
Here $\lb>0$ is called the {\sf viscous damping index}, since the flow $\Phi_{H,\lb}^t$ of (\ref{eq:ode0}) transports the standard symplectic form into a multiple of itself, i.e.
\[
(\Phi_{H,\lb}^t)^*dp\wedge d q= e^{\lb t}dp\wedge d q
\]
for $t$ in the valid domain. That is why system \eqref{eq:ode0} is called {\sf conformally symplectic} \cite{WL} or {\sf dissipative} \cite{LC} in related literatures.

The Hamiltonian satisfying (H1)-(H2) is usually called {\sf Tonelli}. If  (H1)-(H2) is guaranteed,  the {\sf Legendre transformation}
\[
\mathcal{L}_H:T^{\ast}M\rightarrow TM;(x,p)\mapsto(x,H_p(x,p))
\]
is a diffeomorphism and endows a {\sf Tonelli Lagrangian}
\be\label{eq:lag}
L(x,v):=\max_{p\in T_{x}^*\T^n}\langle p, v\rangle-H(x,p)
\ee
of which the maximum is obtained at $\bar p\in T_x^*M $ such that $\cL_H(x,\bar p)=(x, v)$. With the help of $L(x,v)$, we can introduce a variational principle
\[
h_\lb^{a,b}(x,y):=\inf_{\substack{\gamma\in C^{ac}([a,b],\T^n)\\\gamma(a)=x ,\ \gamma(b)=y}}\int_a^b e^{\lb s}L(\gamma(s),\dot{\gamma}(s))ds,\quad (a\leqslant b).
\]
Due to the {\sf Tonelli Theorem} and {\sf Weierstrass Theorem} \cite{Mat}, the infimum is always achievable by a $C^r-$curve $\eta:[a,b]\rightarrow \T^2 $ connecting $x$ and $y$, which satisfies the following {\sf Euler-Lagrange equation}:
\[
\frac d{dt}L_v(\eta,\dot\eta)+\lb L_v(\eta,\dot\eta)=L_x(\eta,\dot\eta).\tag{E-L}
\]
Such an $\eta$ is called a {\sf critical curve} of $h_\lb^{a,b}(x,y)$. Moreover, the {\sf Euler-Lagrange flow} $\phi_{L,\lb}^t$ satisfies $\phi_{L,\lb}^t\circ\mathcal{L}_H=\mathcal{L}_H\circ\Phi_{H,\lb}^t$ in the valid time region. As an equivalent substitute, we can explore the dynamics of (\ref{eq:ode0}) via the variational method.

\subsection{Symplectic aspects of the KAM torus} Let $\alpha=pdx$ be the Liouville form on $T^\ast\T^n$ and
\[
\Sigma:=\Big\{\Big(x,P(x)\Big)\,\Big|\,x\in\T^n, P\in C^{1}(\T^n,\R^n)\Big\}
\]
be a $C^1-${\sf Lagrangian} graph, i.e. $i^*\om|_\Sigma=0$ with $\om=d\alpha=dp\wedge dx$. That implies the symplectic form $\om$ vanishes when restricted to the tangent bundle of $\Sigma$. If so, there must be a $c\in H^1(\T^n,\R)$ and some $u\in C^{2}(\T^n,\R)$ such that
\be\label{eq:1}
P(x)=d_{x}u(x)+c.
\ee
Additionally, if $\Sigma$ is $\Phi^{t}_{H,\lambda}-$invariant for all $t\in\R$, we can prove the following:
\begin{thm}[proved in Sec. \ref{s4}]\label{thm:exact}
$c=\mathbf{0} $. In other words, any $\Phi_{H,\lb}^t-$invariant Lagrangian graph $\Sigma$ has to be exact.
\end{thm}

Usually, the existence of such a $\Phi_{H,\lb}^t-$invariant Lagrangian graph can be guaranteed by certain object with quasi-periodic dynamic in the phase space, i.e. the {\sf KAM torus}:

\begin{defn}[KAM torus]
A homologically nontrivial, $C^1-$graphic, $\Phi_{H,\lb}^t-$invariant set is called a {\sf KAM torus} and denoted by $\cT_\om$, if the dynamic on it conjugates to the $\om-$rotation for some $\om\in\R^n$. In other words,
we can find a $C^1-$ embedding map $K:\T^{n}\rightarrow T^*\T^{n}$ expressed by
\[
K(x):=\Big({\zeta(x)},\eta(x)\Big),\quad\forall x\in\T^n
\]    
with $\zeta(x+m)=\zeta(x)+m$ and $\eta(x+m)=\eta(x)$ for any $m\in\Z^n$, such that the following diagram is commutative for all $t\in\R$:
\be\label{diam-com}
\begin{CD}
x\ (mod\ \Z^n)\in\T^n@>\rho_\om^t>>x+\om t\ (mod\ \Z^n)\in \T^n\\
@VV K(\cdot) V @ VV K(\cdot) V\\
K(x)\in T^*\T^n@>\Phi_{H,\lb}^t>> K(x+\om t)\in T^*\T^n.
\end{CD}
\ee
Consequently we have
\[
\cT_\om=\bigg\{\bigg(x, \eta\big(\zeta^{-1}(x)\big)\bigg)\in T^*\T^n\bigg|x\in\T^n\bigg\}.
\]
Technically, $\om\in H_1(\T^n,\R)$ is called the {\sf frequency} of $\cT_\om$.
\end{defn}

\begin{rmk}
Although this definition does not rely on the choice of $\om$, in most occasions we have to make a special selection of that, to make such a $K(\cdot)$ available. Precisely, we call $\om\in\R^n$ {\sf $\tau-$Diophantine}, if there exists $\alpha>0$ such that
\[
|\langle k,\om\rangle|\geq\frac{\alpha}{\|k\|^\tau},\quad\forall \ k=(k_1,\cdots k_n)\in\Z^n\backslash\{0\}
\]
where $\|k\|=|k_1|+|k_2|+\cdots+|k_n|$.
It has been proved in a bunch of papers, e.g., \cite{CCD1,CCD,CC,Ma}, the existence of the KAM torus with a Diophantine frequency for \eqref{eq:ode0}, by using an analogy of the classical KAM iterations.
\end{rmk}

\begin{thm}[proved in Sec. \ref{s4}]\label{thm:lan}
The KAM torus $\cT_\om$ is naturally  a $\Phi_{H,\lb}^t-$invariant Lagrangian graph.
\end{thm}

\subsection{Variational aspects of the KAM torus}

Following the ideas of Aubry-Mather theory or weak KAM theory in \cite{DFIZ,MS}, 
 we can define the {\sf Lax-Oleinik semigroup} operator
\[
\cT_{t}^{-}: C(\T^n,\R)\rightarrow C(\T^n,\R),\quad t\geq0,
\]
by
\be\label{eq:evo-solu}
\cT_{t}^{-}\psi(x):=\inf_{\substack{\gamma\in C^{ac}([-t,0],\T^n)\\\gamma(0)=x}}\Big\{e^{-\lb t}\psi(\gamma(-t))+\int_{-t}^0 e^{\lb \tau}L(\gamma,\dot\gamma)d\tau\Big\}.
\ee
We can prove that $\cT_{t}^{-}\psi(x)$ is the {\sf viscosity solution} (see Definition \ref{defn:vis}) of the {\sf Evolutionary discounted Hamilton-Jacobi equation (EdH-J)}:
 \be\label{eq:evo-hj}
 \left\{
 \begin{aligned}
 &\partial_tu(x,t)+H(x,\partial_x u)+\lb u=0,\\
 &u(x,0)=\psi(x),\quad t\geqslant 0.
 \end{aligned}
 \right.
 \ee
and
 $u^-(x):=\lim_{t\rightarrow+\infty}\cT^{-}_{t}\psi(x)$ exists uniquely as the viscosity solution of the  {\sf Stationary discounted Hamilton-Jacobi equation (SdH-J)}
\be\label{eq:sta-hj}
H(x,\partial_x u^-(x))+\lb u^-(x)=0,
\ee
no matter which $\psi\in C(\T^n,\R)$ is chosen. By the {\sf Comparison Principle}, the viscosity solution of (\ref{eq:sta-hj}) is unique but usually not $C^{1}$. Therefore, a corollary of Theorem \ref{thm:exact} and Theorem \ref{thm:lan} can be drawn:
\begin{thm}
\label{cor:kam-sol}
For $\lb>0$, each KAM torus $\cT_\om$
gives us the unique classic solution $u_\om^-$ of (\ref{eq:sta-hj}) which satisfies $\cT_\om=\textup{Graph}(d_x u_\om^-)$. Consequently, KAM torus is  unique for equation (\ref{eq:ode0}).
\end{thm}

Based on this Theorem
and the convergence of $\cT_{t}^{-}\psi$ for any initial $\psi$, we can perceive that $\cT_\om$ works as an `attractor' to global action minimizing orbits. 
So we prove the following :




\begin{thm}[$C^1-$convergency]\label{thm:2}
For a $C^r-$smooth Hamiltonian $H(x,p)\in T^*\T^n\rightarrow \R$ satisfying (H1)-(H2) and the associated conformally symplectic system \eqref{eq:ode0}, if there exists a $C^1$-graphic KAM torus 
\[
\cT_\om=\Big\{\Big(x,P(x)\Big)\Big| x \in\T^n\Big\}
\]
with the frequency $\om$, then for any function $\psi\in C(\T^n,\R)$, there exists a constant $C=C(\psi,L,\lb)>0$ such that 
\be\label{eq:main-c1}
\|\cT_t^-\psi(x)-u^-_\om(x)\|_{W^{1,\infty}(\T^n,\R)}\leqslant C e^{-\lb t} 
\ee
for all $t\in [0,+\infty) $, where $P(x)=du_\om^-(x)$ for all $x\in\T^n$. 
\end{thm}

\begin{rmk}
\begin{itemize}
\item This conclusion indicates that, the solution of the Cauchy problem \eqref{eq:evo-hj} converges in exponential speed to the solution of the stationary equation \eqref{eq:sta-hj} w.r.t. the $W^{1,\infty}-$norm, under the prior existence of a KAM torus $\cT_\om$. We will see from the proof (in Sec. \ref{s5}), the semiconcavity of $\cT_t^-\psi(x)$ plays a crucial role in the controlling of $\|\partial_x\cT_t^-\psi(x)\|_{L^\infty}$. However, it remains open to get the convergence speed of the semigroup under norms with higher regularity.


\item We should point out that usually $\cT_\om$ is not a {\sf global attractor} and extra invariant sets may exist in the phase space (see Fig. \ref{fig1}). Nonetheless, the KAM torus is the only destination of all the variational minimal orbits as $t\rightarrow+\infty$,  not the extra invariant sets. 
\end{itemize}
\end{rmk}

\subsection{Is the Lagrangian graph variational stable?}

\begin{que}
Does an invariant Lagrangian graph of a conformally symplectic system still persists as an invariant Lagrangian graph under small perturbation?
\end{que}

The answer to this question seems negative, as is shown in Fig. \ref{fig2}, the dissipative property of (\ref{eq:ode0}) leads to a bifurcation of the global attractor. For $\alpha$ equal to the bifurcate value $0.754...$, the Lagrangian graph is a union of homoclinic orbit and a hyperbolic equilibrium (only Lipschitz smooth). If we reduce the value of $\alpha$, the Lagrangian graph disappears and a non-graphic global attractor comes out, as shown in (a) of Fig. \ref{fig2}.

Since the KAM torus is normally hyperbolic, under small perturbation it persists as a $\Phi_{H,\lb}^t-$invariant $C^1-$graph due to the {\sf Invariant Manifold Theorem} \cite{HPS},  although the dynamic on the perturbed torus may no longer conjugate to a rotation. That implies the persistence of Lagrangian graphs is possible with prior KAM assumption:


\begin{figure}
\begin{center}
\includegraphics[width=7cm]{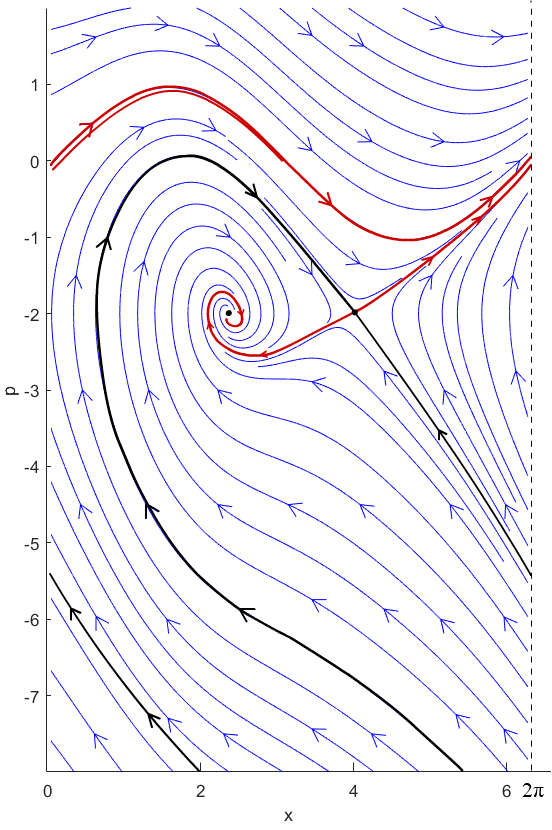}
\caption{Example with a coexistence of KAM torus and fixed points: $H(x,p)=\frac12(p+2)^2-2\sin x+\frac1{\sqrt 2}\cos x+\frac14\cos 2x$, $\lb=\frac1{\sqrt 2}$. The KAM torus equals $\{(x,\sin x)|x\in\T \}$. The maximal global attractor defined in \cite{MS} is marked in red.}
\label{fig1}
\end{center}
\end{figure}


\begin{thm}(proved in Sec. \ref{a1})\label{thm:3}
The KAM torus of a conformally symplectic system keeps to be a $C^1-$Lagrangian graph under small perturbations. 
\end{thm}


\begin{figure}
\begin{center}
\includegraphics[width=12cm]{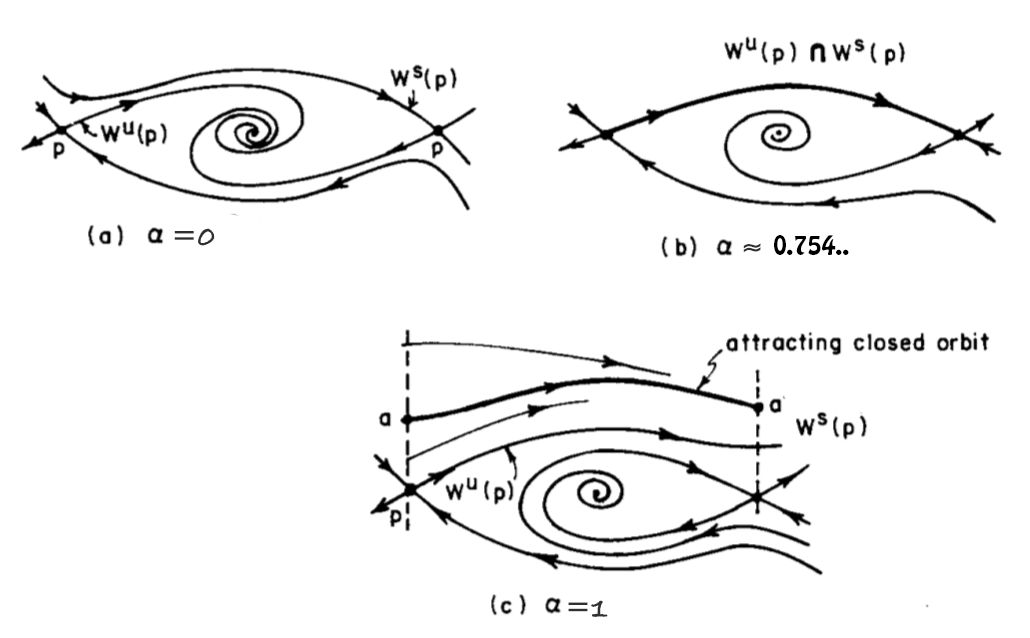}
\caption{Example: $H_\alpha(x,p)=\frac12(p+2\alpha)^2+(\cos x-1)+\big(3-2\alpha +2\sin x-\cos x+\frac1{\sqrt 2}\cos x-\frac14\cos 2x\big)\alpha$, $\lb=\frac1{\sqrt 2}$, $\alpha\in[0,1]$. For $\alpha=0$, the system is actually a dissipative pendulum. For $\alpha=1$, the system has been shown in Fig. \ref{fig1}, with a KAM torus and two fixed points. When $\alpha\approx0.754...$, there would be a bifurcation: we get an invariant torus which is only Lipschitz smooth, and comprises of a hyperbolic equilibrium and its homoclinic orbit. The invariant torus is exact, but the associated viscous solution $u^-(x)$ is only $C^{1,1}$ smooth.}
\label{fig2}
\end{center}
\end{figure}

\noindent{\bf Organization of the article:} The paper is organized as follows: In Sec. \ref{s3}, we give a brief introduction about the the weak KAM theory. In Sec. \ref{s4}, we prove the symplectic properties of the KAM torus, i.e. Theorem \ref{thm:exact} and Theorem \ref{thm:lan}. In Sec. \ref{s5}, we discuss the $C^1-$convergence of the Lax-Oleinik semigroup and prove Theorem \ref{thm:2}. In Sec. \ref{a1} we prove the Lagrangian persistence of the KAM toruus, namely Theorem \ref{thm:3}.
For the consistency of the proof, some longsome and independent conclusions are moved to the Appendix.

\vspace{10pt}

\noindent{\bf Acknowledgements:}  Jianlu Zhang is supported by the National Natural Science Foundation of China (Grant No. 11901560). The authors are grateful to Prof. A Sorrentino for checking the proof of exactness of the Lagrangian graphs and giving constructive suggestions. 

\vspace{20pt}

\section{Weak KAM theory of discounted H-J equations} \label{s3}

\vspace{10pt}

In this section we display a list of definitions and conclusions about the variational principle of system (\ref{eq:ode0}), which can be used in later sections.

%

\begin{defn}[Viscosity solution]\label{defn:vis}
\begin{itemize}
	\item [(1)]A function $u:\T^n\to \R $ is called a {\sf viscosity subsolution} (resp. {\sf viscosity supersolution}) of equation \eqref{eq:sta-hj}, if for every $C^1$ function $\varphi:\T^n \rightarrow \R$ and every point $x_0\in \T^n$ at which $u- \varphi $ reaches a local maximum (resp. minimum) , we have 
	$$
	  H(x_0,\partial_x \varphi(x_0))+\lambda u(x_0) \leqslant 0 \quad ( resp. \geqslant 0 );
	$$
	\item[(2)]A function  $u:\T^n\to \R $ is called a {\sf viscosity solution} of equation \eqref{eq:sta-hj}, if it is both a
viscosity subsolution and a viscosity supersolution.
    \item[(3)] Similarly, a function $U:\T^n\times [0,+\infty)\rightarrow\R$ can be defined by the viscosity subsolution, viscosity supersolution or viscosity solution of equation \eqref{eq:evo-hj}, if aforementioned items holds in the interior region $(0,+\infty)\times M$ for $U$ respectively.

\end{itemize}
	
\end{defn}

\begin{prop}\label{prop-1}
\begin{enumerate}
\item {\sf (Variational principle)}   For each $\psi \in C(\T^n, \R) $, each $x\in \T^n $ and each $t\geqslant 0 $, we can find a $C^2$ smooth curve $\gamma_{x,t}:\tau\in[-t,0]\rightarrow\T^n $ ending with $x$ such that
  \begin{align*}
  	\cT^-_{t}\psi(x)=\, e^{-\lb t}\psi(\gamma_{x,t}(-t))+\int _{-t}^0 e^{\lb \tau } L(\gamma_{x,t}(\tau ),\dot \gamma_{x,t}(\tau) )  d \tau.
  \end{align*}
Moreover, $\cL_H^{-1}(\gamma_{x,t}(\tau),\dot\gamma_{x,t}(\tau))$ satisfies (\ref{eq:ode0}) for $\tau\in(-t,0)$.
 \item {\sf (Pre-compactness)} for any $\psi\in C(\T^n,\R)$ and any $t\geqslant 1$, there exists a constant $K:=K(L,\lambda)>0$ depending only on $L$ and $\lambda$, such that the minimizing curve $\gamma_{x,t}$ achieving $\cT_{t}^-\psi$ satisfies $|\dot\gamma_{x,t}(\tau)|\leqslant K$ for all $\tau\in(-t,0)$.
\item {\sf (Viscosity solution)} Suppose $U_{\psi}(x,t):=\cT_{t}^-\psi(x)$, then it is a viscosity solution of the EdH-J equation (\ref{eq:evo-hj}).
\end{enumerate}
\end{prop}

\begin{proof}

For assertion (1), by taking 
\be\label{eq:h^t}
h_\lambda^t(y,x):=\inf_{\substack{\gamma\in C^{ac}([-t,0],\T^n)\\ \gamma(-t)=y \ \gamma(0)=x }} \int_{-t}^0 e^{\lb \tau}L(\gamma,\dot\gamma)\ d\tau  ,
\ee
we get a simplified expression 
 $$
 \cT_{t}^-\psi(x) =\inf_{y\in \T^n} \{ e^{-\lb t}\psi(y)+ h_\lambda^{t}(y,x)\}.
 $$
 Since the function $y \mapsto e^{-\lb t}\psi(y)+ h_\lambda^{t}(y,x) $ is continuous on $\T^n$, we can find $y_{x,t} \in \T^n$ such that $\cT_{t}^-\psi(x) = e^{-\lb t}\psi(y_{x,t})+ h_\lambda^t(y_{x,t},x) $. Due to the Tonelli Theorem, the infimum of $h^t_\lambda(y,x)$ in \eqref{eq:h^t} is always achievable, and has to be $C^2-$smooth due to the Weierstrass Theorem \cite{Mat}. Hence, we can find a  $C^2$ smooth curve $\gamma_{x,t}:[-t,0] \to \T^n $ with $\gamma_{x,t}(-t)=y_{x,t} $ and $\gamma_{x,t}(0)=x $ such that
$$
h_\lambda^{t}(y_{x,t},x)=\int _{-t}^0 e^{\lb \tau }  L(\gamma_{x,t}(\tau ),\dot \gamma_{x,t}(\tau) ) \ d \tau .
$$

For assertion (2), as we know, for any $\psi\in C^0(\T^n,\R)$ and $t \geqslant 1$,we choose $0<\sigma<t$ such that $ \frac{1}{2} \leqslant \sigma \leqslant 1 $ , due to item (1), there exists a a  $C^2$ smooth curve $\gamma_{x,t}:[-t,0] \mapsto \T^n $ with  $\gamma_{x,t}(0)=x $ and due to   $\cT^-_t$ is a semigroup operator,
\begin{align*}
	\cT^-_{t}\psi (x)=&\,  \inf_{y\in \T^n} \{ e^{-\lb t}\psi (y)+ h_\lambda^{t}(y,x)\} \\
	=&\,\inf_{z\in\T^n } \{ e^{-\lb \sigma } \big( \cT^-_{t-\sigma} \psi  (z) \big) +h_\lambda^{\sigma}(z,x) \}\\
	=&\, e^{-\lb \sigma }  \cT^-_{t-\sigma} \psi  (\gamma_{x,t}(-\sigma)) +h_\lambda^{\sigma}(\gamma_{x,t}(-\sigma ),\gamma_{x,t}(0))  \\
	=&\, e^{-\lb \sigma }  \cT^-_{t-\sigma} \psi  (\gamma_{x,t}(-\sigma))+\int_{-\sigma}^0 e^{\lb \tau}  L(\gamma_{x,t}(\tau ),\dot \gamma_{x,t}(\tau) ) d \tau
\end{align*}
By chosen $\eta:[-\sigma,0] \to \R$ being the straight line connecting $\gamma_{x,t}(-\sigma)$ and $ \gamma_{x,t}(0) $, we have
\begin{align*}
	h_\lambda^{\sigma }(\gamma_{x,t}(\sigma ),\gamma_{x,t}(0)) \leqslant &\, \int_{-\sigma}^0 e^{ \lb \tau}  L(\eta(\tau ),\dot \eta(\tau) ) d \tau \\
 \leqslant &\,  C\int_{-\sigma}^0 e^{\lambda \tau } d \tau
\end{align*}
for a suitable constant $C$ depending only on $L(x,v)$ with $|v|\leqslant diam(\T^n)$. On the other hand, there exists constants $k, l>0$ such that $L(x,v)\geqslant k|v |-l $ for all $(x,v)\in T\T^n$, then
\begin{align*}
	h_\lambda^{\sigma}(\gamma_{x,t}(-\sigma ),\gamma_{x,t}(0)) \geqslant &\,\int_{-\sigma}^0 e^{\lambda \tau}(k|\dot \gamma_{x,t}|-l)d \tau \\
	\geqslant &\, k\int_{-\sigma}^0 e^{\lambda \tau} |\dot \gamma_{x,t}|d \tau-l\int_{-\sigma}^0 e^{\lambda \tau} d \tau.
\end{align*}
Hence, we have 
$$
\int_{-\sigma}^0 |\dot \gamma_{x,t}| d\tau \leqslant \frac{l+C}{ke^{-\lambda \sigma}} \int_{-\sigma}^0 e^{\lambda \tau} d \tau = \frac{l+C}{k\lambda } (e^{\lambda \sigma }-1).
$$
There always exists a $t'\in [-\sigma,0]$ such that $|\dot \gamma_{x,t}(t')|\leqslant \frac{l+C}{k\lambda } \frac{e^{\lambda \sigma}-1}{\sigma} $. Since $ \gamma_{x,t}$ satisfies the (E-L), then $|\dot \gamma_{x,t}(\tau)| $ is uniformly bounded on $\tau\in [-\sigma,0]$.

 Choose $-t\leqslant -\sigma_N\leqslant -\sigma_{N-1}\leqslant \cdots \leqslant -\sigma_1\leqslant -\sigma_0=-\sigma\leqslant 0$ such that $\frac{1}{2}\leqslant \sigma_i-\sigma_{i+1}\leqslant 1$, then for any $1\leqslant i\leqslant N$ we get
$$
\cT^-_{t-\sigma_{i-1}}\psi (\gamma_{x,t}(-\sigma_{i-1}))=e^{\lb(\sigma_{i-1}-\sigma_i) }  \cT^-_{t-\sigma_i} \psi  (-\gamma_{x,t}(-\sigma_{i} ))+\int_{-\sigma_i}^{-\sigma_{i-1}} e^{\lb \tau}  L(\gamma_{x,t}(\tau ),\dot \gamma_{x,t}(\tau) ) d \tau,
$$
the same scheme as above still works. So $|\dot \gamma_{x,t}(\tau)| $ is uniformly bounded on $\tau\in [-t,0]$.
\medskip

For assertion (3), it's a classical conclusion in the Optimal Control Theory (e.g. \cite[Chapter III]{Ba}) that $U_{\psi}(x,t):=\cT_{t}^-\psi(x)$ is a continuous viscosity solution of \eqref{eq:sta-hj}. Here we give a sketch: 

To prove $ U_{\psi}(x,t)=\cT_{t}^-\psi(x)$ is a subsolution; for any fixed $(x_0,t_0)\in   \T^n \times (0,+\infty)$. Let $\varphi$ be a $C^1$ test function such that $(x_0,t_0)$ is a local maximal point of $U_{\psi}-\varphi$ and $U_{\psi}(x_0, t_0)=\varphi (x_0,t_0 ) $. That is
\be \label{eq:subsolution}
\varphi(x_0,t_0)-\varphi(x,t)\leqslant U_{\psi}(x_0,t_0)-U_{\psi}(x,t),\quad (x,t)\in W,
\ee
	with $W$ be an open neighborhood of $(x_0,t_0)$ in $\T^n \times (0,+\infty)$.
	
	Due to item(1) and $\cT_t^-$ is a semigroup operator, for any differentiable point $x_2 \in W$, we have that $\cT_{t_2-t_1}^- \circ \cT_{t_1}^- \psi(x_2) = \cT_{t_2}^- \psi(x_2)$ for any $t_1<t_2$ in $W$ which implies that 
	\begin{align*}
		e^{\lb t_2} U_{\psi}(x_2,t_2)- e^{\lb t_1} U_{\psi}(x_1,t_1)\leqslant \,\int^{t_2}_{t_1}  e^{\lb \tau }  L(\gamma,\dot{\gamma} )  d \tau.
	\end{align*}
for any $C^1$ curve $\gamma:[t_1,t_2]\to \T^n$ connecting $x_1$ to $x_2$. By \eqref{eq:subsolution} it follows
	\begin{align*}
		\varphi(t_2,x_2)-\varphi(t_1,x_1) \leqslant \int^{t_2}_{t_1}  e^{\lb (\tau-t_2) }  L(\gamma,\dot{\gamma}) d \tau  -(1 -e^{\lb( t_1-t_2) })U_{\psi}(x_1,t_1).
	\end{align*}
	By letting $|t_2-t_1|\to 0$, this gives rise to
	\begin{align*}
		\partial_t \varphi(x_0,t_0)+\partial_x\varphi(x_0,t_0)\cdot \dot \gamma(t_0) \leqslant L(x_0,\dot \gamma(t_0) )-\lb U_\psi (x_0,t_0) 
	\end{align*}
	As an application of Fenchel-Legendre dual, we obtain
	\begin{align*}
		\partial_t \varphi(x_0,t_0)+H(x_0,\partial_x\varphi(x_0,t_0))+\lb U_\psi (x_0,t_0) \leqslant 0,
	\end{align*}
	which shows that $U_{\psi}$ is a subsolution.

Now we turn to the proof that $U_{\psi}$ is a supersolution. Let $\varphi$ be a $C^1$ test function such that $(x_0,t_0)$ is a local minimal point of $u-\varphi$ and $U_{\psi}(x_0,t_0)=\varphi(x_0,t_0)  $ . That is
	\begin{align*}
		\varphi(x_0,t_0)-\varphi(x,t)\geqslant U_{\psi}(x_0,t_0)-U_{\psi}(x,t),\quad (x,t)\in V,
	\end{align*}
	with $V$ be an open neighborhood of $(x_0,t_0)$ in $\T^n$. There exists a $C^2$ curve $\xi:[t, t_0]\to V$ with $\xi(t_0)=x_0$ such that
	\begin{align*}
		U_{\psi}(\xi(t_0),t_0)-U_{\psi}(\xi(t),t)=\int^{t_0}_{t} L(\xi,\dot{\xi} )-\lb U_{\psi}(\xi(s),s)  \ ds.
	\end{align*}
	Hence
	\begin{align*}
		\varphi(x_0,t_0)-\varphi(\xi(t),t )\geqslant\int^{t_0}_{t} L(\xi,\dot{\xi} )-\lb U_{\psi}(\xi(s),s)  \ ds,
	\end{align*}
	It follows that
	\begin{align*}
		\partial_t \varphi (x_0,t_0)+\partial_x\varphi (x_0,t_0)\cdot\dot{\xi}(t_0)\geqslant L(x_0,\dot{\xi}(t_0) )-\lb U_\psi (x_0,t_0) ,
	\end{align*}
	which implies
	\begin{align*}
		\partial_t \varphi(x_0,t_0)+H(x_0,\partial_x\varphi(t_0,x_0))+\lb  U_\psi(x_0,t_0) \geqslant 0.
	\end{align*}
	So we finally finish the proof.
\end{proof}

\begin{prop}\label{prop-2}
 \begin{enumerate}
 \item {\sf (Expression)} \cite[Theorem 6.1]{DFIZ}  For all $\psi\in C^0(M,\R)$, the limit of $\cT_{t}^-\psi(x)$ exists and can be explicitly expressed, i.e.
 \be\label{eq:sta-solu}
\lim_{t\rightarrow+\infty}\cT_{t}^-\psi =\inf_{\substack{\gamma\in C^{ac}((-\infty,0],\T^n)\\\gamma(0)=x}}\int_{-\infty}^0 e^{\lb\tau}L(\gamma,\dot\gamma)d\tau.
\ee
Since the limit is independent of $\psi$, we can denote it by $u^-(x)$.
\item {\sf (Domination)} \cite[Proposition 6.3]{DFIZ} For any absolutely continuous curve $\gamma:[a,b]\rightarrow\T^n$ connecting  $x, y\in\T^n$, we have
\be
e^{\lb b}u^-(y)-e^{\lb a}u^-(x)\leqslant \int_a^b e^{\lb\tau}L(\gamma,\dot\gamma)d\tau.
\ee
\item {\sf (Calibration)}\cite[Proposition 6.2]{DFIZ} For any $x\in\T^n$, there exists a $C^r$ {\sf backward calibrated }curve $\gamma_{x}^-:(-\infty,0]\rightarrow\T^n$ ending with $x$, such that for all $s\leqslant t\leqslant 0$, we have
\be
e^{\lb t}u^-(\gamma_{x}^-(t))-e^{\lb s}u^-(\gamma_{x}^-(s))=\int_s^t e^{\lb\tau}L(\gamma_{x}^-,\dot\gamma_{x}^-)d\tau.
\ee
Similarly, $\cL_H^{-1}(\gamma_{x}^-(\tau),\dot\gamma_{x}^-(\tau))$ satisfies (\ref{eq:ode0}) for $\tau\in(-\infty,0)$.
\item {\sf (Pre-compactness)}\cite[Proposition 6.2]{DFIZ} There exists a constant $K>0$ depending only on $L$,  such that the minimizing curve $\gamma_{x}^-$ of $u^-(x)$ satisfies $|\dot\gamma_{x}^-(\tau)|\leqslant K$, for all $\tau\in(-\infty,0)$.
\item  \cite[Proposition 5,6]{MS} Along each calibrated curve $\gamma_{x}^-:(-\infty,0]\rightarrow \T^n$, we have
\[
\lb u^-(\gamma_x^-(s))+H(\gamma_x^-(s), d_xu^-(\gamma_x^-(s)))=0,\quad\forall s<0.
\]
\item {\sf ($C^0-$convergent speed)}
	 Let $U_{\psi}(x,t):=\cT_t^-\psi$ be the viscosity solutions of the equation  \eqref{eq:evo-hj}, there exists a constant $C_1=C_1(\psi,L,\lambda)$, such that
	$$
	\big\| U_{\psi}(x,t)-u^-(x) \big\| \leqslant  C_1e^{-\lambda t},\quad\forall t\geqslant 1.
	$$
\item {\sf (Stationary solution)} $u^-(x)$ is a viscosity solution of (\ref{eq:sta-hj}).
 \end{enumerate}
 \end{prop}

 \begin{proof}
As direct citations, we have marked the exact references for the first five items of this Proposition. For item (6), due to the expression of $u^-(x)$ in \eqref{eq:sta-solu} and item (3), there must exist an absolutely continuous curve $\gamma_x^-:(-\infty, 0]\rightarrow \T^n $ with $\gamma_x^- (0)=x$ such that 
$
u^-(x)= \int_{-\infty}^0 e^{\lb \tau} L(\gamma_x^- ,\dot\gamma_x^- ) d\tau.
$
Then  we have
\begin{align*}
	&\, U_{\psi}(x,t)-u^-(x)\\
	\leqslant &\, e^{-\lb t}\psi(\gamma_x^-(-t))+\int_{-t}^0 e^{\lb \tau } L(\gamma_x^-,\dot\gamma_x^-) d\tau  -\int_{-\infty}^0 e^{\lb \tau} L(\gamma_x^-,\dot\gamma_x^-) d\tau \\
	 = &\, e^{-\lb t}\psi(\gamma_x^-(-t)) -\int_{-\infty}^{-t} e^{\lb \tau} L(\gamma_x^-,\dot\gamma_x^-)d\tau   \\
	 \leqslant &\, e^{-\lb t}\psi(\gamma_x^-(s)) -\min_{(x,v)\in T\T^n} L(x,v)  \int_{-\infty}^{-t} e^{\lb \tau}d\tau  \\
	 \leqslant &\, e^{-\lb t}\Big[   ||\psi||_{C^0} +\frac{1}{\lb}\min_{(x,v)\in T\T^n} L(x,v)  \big] \\
	 =&\, \widetilde C_1 e^{-\lb t},
\end{align*}
where $\widetilde C_1$ is a constant depending on $||\psi||_{C^0},\lambda$ and $\min_{T\T^n} L(x,v)$.
   On the other hand, there is an absolutely continuous curve $\gamma_{x,t}:[-t,0]\rightarrow \T^n $ with $\gamma^-_x(0)=x$ such that $U_{\psi}(x,t)$ attains the infimum in the formula \eqref{eq:evo-solu}, define $\xi:(-\infty,0] \rightarrow \T^n $ by $ \xi(\tau)=\gamma_{x,t}(\tau)$ for $\tau \in [-t,0] $ and $\xi(\tau)\equiv \gamma_{x,t}(-t) $ for $ \tau \leqslant -t $, it follows that $\xi $ is an absolutely continuous curve with $\xi(0)=x$ and
   \begin{align*}   	
     &\, u^-(x)-U_{\psi}(x,t) \\
   	\leqslant &\,\int_{-\infty}^0 e^{\lb \tau} L(\xi,\dot\xi) d\tau  -e^{-\lb t} \psi(\xi(-t))  - \int_{-t}^0 e^{\lb \tau} L(\xi,\dot\xi)  d\tau \nonumber\\
   	 \leqslant &\, e^{-\lb t} |\psi(\xi(-t))| + \int_{-\infty}^{-t} e^{\lb \tau} L(\xi,\dot\xi)  d\tau \nonumber \\
   	  \leqslant &\,  e^{-\lb t} \big[ ||\psi||_{C^0}+ \frac{1}{\lb}\max_{x\in\T^{n}}|L(x,0) |   \big]\nonumber\\
   	 =&\, \widetilde C_2 e^{-\lb t} \nonumber ,
   \end{align*}
   with $\widetilde C_2$ being a constant depending on $||\psi||_{C^0},\lambda $ and $\max_{x\in\T^{n}}|L(x,0) |$. Combining previous two inequalities we prove this item.

\medskip
 
For item (7), the proof is similar with item (4) of Proposition \ref{prop-1}.
\end{proof}

\medskip

\section{Exactness of the KAM torus}\label{s4}

\vspace{10pt}

\noindent{\it Proof of Theorem \ref{thm:exact}:}
The invariance of $\Sigma$ implies for any $x\in\T^{n}$,
\[
\Phi^{t}_{ H,\lb}\Big(x,P(x)\Big)=\Big(x(t),P\big(x(t)\big)\Big),\quad\forall t\in\R.
\]
Due to \eqref{eq:ode0}, for any $x\in\T^n$,
\be
-H_{x}(x,P(x))-\lb P(x)&=&\dot{p}(0)\nonumber\\
&=&\,\frac{dP\big(x(t)\big)}{dt}\bigg|_{t=0}=d_{x}P(x)\cdot\dot{x}(0)\nonumber\\
&=&\langle d_{x}P(x), H_p(x,P(x))\rangle.
\ee
On the other side, we define
\[
G(x):=\lb u(x)+H(x,P(x))\in C^1(\T^{n},\R),
\]
which satisfies
\be
d_xG(x)&=& \lb  d_xu(x)+d_xH(x,P(x)) \nonumber\\
&=&\,\langle d_{x}P(x), H_p(x,P(x))\rangle dx+\lb d_{x}u(x)dx+H_{x}(x,P(x))dx\nonumber\\
&=&\,\lb[d_{x}u(x)-P(x)]dx=\,-\lb cdx
\ee
since $P(x)=c+d_x u(x)$. We can read through previous equality for $i=1,...,n$,
\[
\partial_{x_{i}}G(x)+\lb c_{i}=0.
\]
By integrating the above equality w.r.t. $x_{i}$ over $\T$, then $c_{i}=0$ for $i=1,...,n$.\qed \\

\noindent{\it Proof of Theorem \ref{thm:lan}:} It suffices to show that $\Omega (T\cT_\om, T\cT_\om)= 0$, which is equivalent to show $K^*\Om\big|_{T\T^n}= 0$. Recall that
\ben
(\Phi_{H,\lb}^t\circ K)^*\Om =K^*(\Phi_{H,\lb}^t)^*\Om=e^{\lb t} K^*\Om.
\een
On the other side, $K\circ\rho_\om^t=\Phi_{H,\lb}^t\circ K$, which implies
\[
(K\circ\rho_\om^t)^*\Om=(\rho_\om^t)^* K^*\Om.
\]
Combining these two equalities we get
\[
(\rho_\om^{-t})^* K^*\Om=e^{-\lb t}K^*\Om,\quad \forall \ t\in\R_+.
\]
Since $\cT_\om$ is $\Phi_{H,\lb}^t-$invariant, and $\lim_{t\rightarrow+\infty}(\rho_\om^{-t})^* K^*\Om=0$, we prove $K^*\Om=0$.\qed
\vspace{10pt}

\section{ $W^{1,\infty}-$convergence speed of the Lax-Oleinik semigroup}\label{s5}

\vspace{10pt}

\subsection{Semiconcave functions with linear modulus}\label{a5}

\begin{defn}[Hausdorff metric]\label{defn:Hausdorff metric}
	Let $(X,d)$ be a metric space and $\mathcal{K}(X)$ be the set of non-empty compact subset of $X$. The Hausdorff metric $d_H$ induced by $d$ is defined by
	$$
	d_H(K_1,K_2)=\max \big\{ \max_{x\in K_1}d(x,K_2), \max_{x\in K_2}d(K_1,x) \big\}, \quad \forall K_1,K_2\in \mathcal{K}(X)
	$$
\end{defn}

\begin{defn}[SCL]
Let $\cU \subset \R^n $ be a open set. A function $f:\cU\to \R$ is said to  be {\sf semiconcave with linear modulus (SCL for short)} if there exists a constant $C>0$ such that
$$
\lambda f(x)+(1-\lambda)f(y)-f(\lambda x+(1-\lambda )y )\leqslant \frac{C}{2} \lambda(1-\lambda )|x-y|^2 \quad  \forall x,y \in \cU  , \ \forall \lambda\in [0,1].
$$
\end{defn}
\begin{defn}\label{def:D+D-}
Assume $f\in C(\cU,\R)$, for any $x \in \cU $, the closed convex set
	$$
	D^+  f(x)=\Big\{ \eta \in  T^*\cU  : \limsup_{|h |\to 0 } \frac{f(x+h )-f(x)-\langle \eta ,h \rangle }{|h|} \leqslant 0  \Big\}
	$$
	$$
	\Big( \ \text{resp.} \ D^- f(x)=\Big\{  \eta \in  T^*\cU : \limsup_{|h |\to 0 } \frac{f(x+h )-f(x)-\langle \eta ,h \rangle }{|h|} \geqslant 0  \Big\} \Big)
	$$
	is called the {\sf super-differential (resp. {\sf sub-differential}) set} of $f$ at $x$.
\end{defn}
\begin{defn}
	Suppose $f: \cU \to \R $ is local Lipschitz. A vector $p\in  T^*\cU$ is called a {\sf reachable gradient} of $u$ at $x \in  \cU$ if a sequence $\{ x_k \}_{k\in\N} \subset \cU \backslash \{x\} $ exists such that $f$ is differentiable at $x_k$ for each $k\in \N$, and 
	$$
	 \lim_{k \to \infty}x_k=x,\quad \lim_{k \to \infty }d_xf(x_k)=p.
	$$
	The set of all reachable gradients of $f$ at $x$ is denoted by $D^* f(x)$.
\end{defn}

\begin{lem}\cite[Theorem.3.1.5(2)]{CANNARSA}
$f:\cU\subset\R^d\rightarrow \R$ is a SCL, then $D^+f(x)$ is a nonempty compact convex set  for any $x\in\cU$ .
\end{lem}

\begin{thm}\cite[Theorem.3.3.6]{CANNARSA}\label{D^+ convex}
	Let $f:\cU \to \R $ be  a semiconcave function. For any $z\in \cU$,
	$$
	D^+ f(z)=co\footnote{convex hull, i.e. the smallest convex set containing the given set} D^* f(z),
	$$
	i.e. any element in $D^+f(z)$ can be expressed as a convex combination of elements in $D^*f(z)$. As a corollary, ex$(D^+f(z))\subset D^*f(z)$, i.e. any extremal element of $D^+f(z)$ has to be contained in $D^*f(z)$.
\end{thm}


\begin{thm}\label{D^+ semiconcave}\cite[Th. 5.3.8]{CANNARSA}
For any fixed $t>0$, the viscosity solutions $U_{\psi}(x,t):=\cT_{t}^-\psi(x)$ (resp. $u^-(x)$) of (\ref{eq:evo-hj}) (resp. (\ref{eq:sta-hj})) is $SCL_{loc}$ (resp. SCL) on $\T^n \times (0,+\infty) $ (resp. $\T^n$).
\end{thm}

\begin{thm}\label{D^* 1-1}
For any $ x\in \T^n$ (resp. $(x,t)\in \T^n \times(0,+\infty)$) and  $p\in D^*u^-(x)$ (resp. $ (p_t,p_x) \in  D^*U_{\psi}(x,t)$),there is a minimal curve $\gamma_{x}^-:(-\infty,0]\rightarrow \T^n$ (resp. $\gamma_{x,t}:[-t,0]\rightarrow \T^n$) satisfying 
\[
u^-(x)=\int_{-\infty}^0 e^{\lb\tau}L(\gamma_{x}^-(\tau),\dot\gamma_{x}^-(\tau))d\tau.
\]
\begin{align*}
\bigg(	resp.\quad U_{\psi}(x,t)= e^{-\lb t}\psi(\gamma_{x,t}(-t))+\int _{-t}^0 e^{\lb \tau}  L(\gamma_{x,t}(\tau ),\dot \gamma_{x,t}(\tau) )   d \tau\bigg)
\end{align*}
and
\[
\lim_{s\rightarrow 0_-}\dot\gamma_x^-(s)=\frac{\partial H}{\partial p}(x,p) \quad
\bigg(	resp.\quad\lim_{s\rightarrow 0_-}\dot\gamma_{x,t}(s)=\frac{\partial H}{\partial p}(x,p_x)\bigg).
\]
Conversely, for any calibrated curve $\gamma_x^-:(-\infty,0]\rightarrow \T^n$ (resp. $\gamma_{x,t}:[-t,0]\rightarrow \T^n$) ending at $x$, the left derivative at $s=0$ (resp. $s=0$) exists and satisfies 
\[
\lim_{s\rightarrow 0_-}L_v(\gamma_x^-(s),\dot\gamma_x^-(s))\in D^*u^-(x)
\]
\[
\bigg(	resp.\quad
     \begin{pmatrix}
      \lim_{s\rightarrow 0_-}L_v(\gamma_{x,t}(s),\dot\gamma_{x,t}(s))    \\
        -\lim_{s\rightarrow 0_-}\lb U_\psi(\gamma_{x,t}(s),s)+H(\gamma_{x,t}(s),L_v(\gamma_{x,t}(s),\dot\gamma_{x,t}(s)))
\end{pmatrix}
        \in D^*U_\psi(x,t)\bigg).
\]
\end{thm}
\begin{proof}
If  $u^-$ is differentiable at $x$, by item (3) of Proposition \ref{prop-2}, there exists a unique $\lb-$calibrated curve $\gamma_{x}^-:(-\infty,0]\rightarrow \T^n$ ending with $x$, such that $u^-(\gamma_x^-(s))$ is differentiable for any $ s\in(-\infty,0) $, which implies
\begin{equation}\label{le-dual}
(x,\lim_{\tau\rightarrow 0_-}\dot{\gamma}^-_{x}(\tau))=\cL_H^{-1}(x,d_{x}u^{-}(x)).
\end{equation} 
Equivalently, $(\xi(s),p(s)):=(\gamma^-_{x}(s),d_{x}u^{-}(\gamma_x^-(s) )  ) $ solving
\be\label{ODE-discount}
	\begin{split}	
	\begin{cases}
		\dot \xi(s)=\partial_p  H(\xi (s),p(s) ) , \\
		\dot p(s) =-\partial_x H(\xi(s),p(s) )-\lb  p(s) 
	\end{cases}
	\end{split}
\ee
for $s\in(-\infty,0]$.

If $x\in  \T^n $ is a non-differentiable point of $u^-$, then for any $p_x\in D^* u^-(x)$, there exists a sequence $\{x_k \}_{k\in\N} $ of differentiable points of $u^-$ converging to $x$, such that $p=\lim_{k\to \infty} d_xu^-(x_k)$. Due to item (5) of Proposition \ref{prop-2} we have
$$
\lambda u^-(x_k) +H(x_k,d_x u^-(x_k)) =0
$$
and there exists minimizing curves $\{\gamma_{x_k}^-\}_{k\in\N}$ solving \eqref{ODE-discount}, such that by
letting $k \to \infty $, we get
\[
\lb u^-(x)+ H(x, p_x)=0 .
\]

Due to the uniqueness of the solution of  \eqref{ODE-discount}, the limit curve $\gamma_{x}^-:(-\infty,0]\rightarrow \T^n$ of the sequence of minimizing curves $\{\gamma_{x_k}^-\}_{k\in\N}$ has to be unique as well, 
with the 
terminal conditions $\gamma_x^- (0)=x, \ p(0)=p_x $.
This proves that the correspondence 
\[
\Upsilon: p_x \in D^*u^-(x)\to \gamma_{x}^-(\tau)\big|_{\tau\in(-\infty,0]}
\]
 is injective.

Now we prove the other direction.
if $\gamma\in C^{ac}((-\infty,0],\T^n)$ is $\lambda$-calibrated by $u^{-}$ with $\gamma(0)=x$, due to item (4) and (7) of Proposition \ref{prop-2} $u^{-}$ is differentiable at $\gamma(s)$ and $\gamma$ is actually $C^{2}$ for any $s\in(-\infty,0)$. Therefore, by \eqref{le-dual} , setting 
\begin{align*}
	 & p_x=\lim_{s\rightarrow 0_{-}}d_{x}u^{-}(\gamma(s))=\lim_{s\rightarrow 0_{-}}L_v(\gamma(s),\dot \gamma(s)),
\end{align*}
there holds $p_x \in D^{\ast}u(x)$. By a similar analysis the conclusion can be proved for $U_\psi(x,t)$, or see \cite[Th.6.4.9]{CANNARSA} for a direct citation.
\end{proof}

\medskip

\subsection{Proof of Theorem \ref{thm:2}}
Based on aformentioned preparations, we turn to the proof of Theorem \ref{thm:2}. Recall that the KAM torus $\cT_\om$ is the graph of $d u^-$ (due to Theorem \ref{cor:kam-sol}), where $u^-(x)$ is the unique $C^2-$classic solution of (\ref{eq:sta-hj}). On the other side, $U_{\psi}(x,t)$ is SCL$_{loc}$ w.r.t. $(x,t) \in \T^n \times (0,+\infty)$, then $D^+U_{\psi}(x,t)$ has to be a compact convex set. Now we assume $t\geqslant 1$, then for any $ (p_x,p_t) \in D^*U_{\psi}(x,t)$, due to Proposition \ref{D^* 1-1} there is a unique minimizer curve  $\gamma_{x,t} $ with $\gamma_{x,t} (0)=x$ such that 
\begin{align*}
\quad U_{\psi}(x,t)= e^{-\lb t}\psi(\gamma_{x,t}(0))+\int _{-t}^0 e^{\lb \tau}  L(\gamma_{x,t}(\tau ),\dot \gamma_{x,t}(\tau) )   d \tau
\end{align*}
with $p_x= \lim_{\tau \to 0_- } d_x  U_{\psi}(\gamma_{x,t}(\tau ),\tau )  $.  Moreover, the following properties of $u^-$ and $U_{\psi}$ can be proved easily:

\begin{lem} \label{lem: converange prop} 
\begin{itemize}
\item[(1)] For any fixed $x \in \T^n $, we have
\ben
	d_H(d_x u^-(x),\Pi_x D^+ U_{\psi}(x,t)) 
	=\, \max_{\substack{ p_x \in \Pi_x D^*U_{\psi}(x,t) } } | d_x u^-(x) -p_x  | \een	
	where $\Pi_x: T^*_{(x,t)}( \T^n\times \T) \to T_x^* \T^n $ is the standard projection.
\item[(2)] There exists a constant $C_2(\psi,L, \lambda)$ depending only on $\psi$ and $L$, such that 
\ben
	|d_x u^-(x)-p_x |
	\leqslant  C_2(\psi, L, \lambda) \lim_{\tau \to 0_- } |\dot \gamma_x^-(\tau ) -\dot \gamma_{x,t}(\tau ) | 
\een
\item[(3)] There exists a constant $A:=A(\psi,L, \lambda)$ such that 
\be\label{compact}
\left\{
\begin{aligned}
	&\, \big| \dot \gamma_{x,t}(\tau) \big|, \big| \ddot \gamma_{x,t}(\tau)\big|  \leqslant A \\
	&\, \big|\dot \gamma_x^-(\tau) \big|, \big| \ddot \gamma_x^-(\tau) \big|  \leqslant A
	\end{aligned}
	\right. \quad \quad \tau\in(-t,0),\quad x\in \T^n.
\ee
	\end{itemize}
\end{lem}

%
\begin{proof}
(1) Due to Lemma \ref{D^+ convex}, $ D^+  U_{\psi}(x,t)=coD^*  U_{\psi}(x,t)$ then by Definition \ref{D^* 1-1} , it obtain 
  \begin{align*}
  	&\, d_H(d_x u^-(x),\Pi_x D^+ U_{\psi}(x,t)) 
	= \max_{\substack{ p_x \in \Pi_x D^+U_{\psi}(x,t) } } | d_x u^-(x) -p_x  |\\
	 =&\, \max_{\substack{ p_x \in co \Pi_x D^*U_{\psi}(x,t) } } | d_x u^-(x) -p_x  |
	=\max_{\substack{ p_x \in \Pi_x D^*U_{\psi}(x,t) } } | d_x u^-(x) -p_x  |.
  \end{align*}

(2) Since $\gamma_x^-(0)=\gamma_{x,t}(0)=x$, due to Theorem \ref{D^* 1-1} and item (5) of Proposition \ref{prop-2},  we obtain that 
\begin{align*}
	p_x= \lim_{\tau \to 0_- } d_x  U_{\psi}(\gamma_{x,t}(\tau ),\tau ) =\lim_{\tau \to 0_- } L_v( \gamma_{x,t}(\tau ),\dot \gamma_{x,t}(\tau )) , \\
	d_x u^-(x)= \lim_{\tau \to 0_- } d_x  u^-(\gamma_{x}^-(\tau ),\tau )=\lim_{\tau \to 0_- } L_v( \gamma_{x}^-(\tau ),\dot \gamma_{x}^-(\tau )) .
\end{align*}
Due to item(2) of Proposition \ref{prop-1} and item(4) of Proposition \ref{prop-2}, we have
\begin{align*}
	 |d_x u^-(x)-p_x |\leqslant & \, \lim_{\tau \to 0_- } | L_v(\gamma_x^-(\tau) ,\dot \gamma_x^-(\tau ))-L_v( \gamma_{x,t}(\tau ),\dot \gamma_{x,t}(\tau )) |\\
	 \leqslant & \, C_2(\psi, L, \lambda) \lim_{\tau \to 0_- } |\dot \gamma_x^-(\tau ) -\dot \gamma_{x,t}(\tau ) | .
\end{align*}

(3)  Denote that $p(\tau)=L_v(\gamma_{x,t}(\tau),\dot \gamma_{x,t}(\tau))$, due to dissipative equation \eqref{eq:ode0}, we have 
\begin{align*}
	&\,\ddot{\gamma}_{x,t}(\tau)= \frac{d}{dt}H_p\big(\gamma_{x,t}(\tau),p(\tau) \big)
	=H_{px}(\gamma_{x,t}(\tau),p(\tau)) \cdot H_p(\gamma_{x,t}(\tau),p(\tau) ) \\  &\, \quad \quad\quad \quad\quad \quad\quad \quad\quad  +H_{pp}(\gamma_{x,t}(\tau),p(\tau)) (-H_x(\gamma_{x,t}(\tau),p(\tau))-\lambda p(\tau)  ) \quad \tau\in(-t,0).
\end{align*}
By item(2) of Proposition \ref{prop-1} and item(4) of Proposition \ref{prop-2}, for any $\tau\in(-t,0)$,   $(\gamma_{x,t}(\tau),\dot \gamma_{x,t}(\tau) )$ and $(\gamma_{x}^-(\tau),\dot \gamma_{x}^-(\tau) )$  fall in both a compact set of $\T^n$. Hence, there exists a constant $A:=A(\psi,L, \lambda)$ such that $\big| \dot \gamma_{x,t}(\tau) \big|, \big| \ddot \gamma_{x,t}(\tau)\big|  \leqslant A$ and $\big| \dot \gamma_{x,t}(\tau) \big|, \big| \ddot \gamma_{x,t}(\tau)\big|  \leqslant A$ for any $\tau\in (-t,0)$, even for $\tau\in[-t,0]$ if we consider the unilateral limit. This completes the proof.
 \end{proof}

\medskip

 Now we define a substitute Lagrangian 
\beq\label{eq:rec-lan}\tag{$\star$}
l(x,w ):=L(x,w )-\lambda u^- (x)-\langle w, d_xu^-(x)\rangle 
\eeq
on $(x,w)\in T\T^n $. Due to item (3) of Proposition \ref{prop-2}, there exists a $\lb-$calibrated curve $\gamma_x^-:(-\infty,0]\rightarrow \T^n$, of which 
\[
u^-(x)-e^{-\lb t}u^-(\gamma_x^-(-t))=\int_{-t}^0 e^{\lb\tau}L(\gamma_x^-,\dot\gamma_x^-)d\tau
\]
and
\begin{align*}
	&\, u^-(x)-e^{-\lb t}u^-(\gamma_x^-(-t) )=\int_{-t}^0 \frac{d}{d\tau}\big[ e^{\lambda \tau }u^-(\gamma_x^-(\tau )) \big] \ d \tau \\
	=&\,\int_{-t}^0 e^{\lambda \tau }\Big[ \dot \gamma_x^-(\tau) \cdot d_x u^- + \lambda u^- \Big] \ d \tau.
\end{align*}
Therefore, we have
\be
\int_{-t}^0 e^{\lb\tau} l(\gamma_x^-(\tau),\dot \gamma_x^-(\tau)  ) \ d \tau =0
\ee
 and 
\be
\int_{-t}^0 e^{\lb\tau} l(\eta(\tau),\dot \eta(\tau) )\ d \tau \geqslant 0
\ee
  for  any $C^1-$curve $\eta:[-t,0]\rightarrow\T^n$, due to item (2) of Proposition \ref{prop-2}. 
\begin{lem}
	Suppose $t\geqslant \sigma$ for some $\sigma>0$, then there exists a constant $\alpha_0:=\alpha_0(\psi,L,\lambda) $ depending only on $L,\psi,\lambda$ such that
	\be\label{eq:estimate l}
	\left\{
	\begin{aligned}
		  \int^{0}_{-\sigma} e^{\lambda \tau }  l ( \gamma_{x,t}(\tau),\dot \gamma_{x,t}(\tau))\ d \tau   \leqslant  &\, 2  C_1(\psi,L, \lambda ) \ e^{- \lambda t}, \\
	 \int^{0}_{-\sigma} e^{\lambda \tau }  l ( \gamma_{x,t}(\tau),\dot \gamma_{x,t}(\tau)  )\ d \tau    \geqslant &\, \alpha_0(\psi,L,\lambda) \int^{0}_{-\sigma} e^{\lambda \tau }  \big|\dot \gamma_{x,t}(\tau)- \dot \eta_{x,t} (\tau)\big|^2\ d \tau,
	\end{aligned}
	\right.
	\ee
	where 
	\be\label{eq:curve eta}
	\eta_{x,t}(r):=x-\int^0_r \frac{\partial H}{\partial p}(\gamma_{x,t}(\tau) , d_x u^-(\gamma_{x,t}(\tau) )) \ d \tau 
	\ee
	for any $r\in(- \sigma , 0) $.
\end{lem}
\begin{proof}
 By the definition of the solution semigroup $\cT_{t}^- \psi(x)=U_{\psi}(x,t)$ and item (2) of Proposition \ref{prop-1}, there exists a $C^2$ curve $\gamma_{x,t}:[s,t]\mapsto \T^n$ such that
$$
 U_{\psi}(x,t )=e^{-\lambda \sigma  } U_{\psi}(\gamma_{x,t }(-\sigma),-\sigma) + \int_{-\sigma }^{0} e^{\lambda \tau } L( \gamma_{x,t}(\tau),\dot \gamma_{x,t}(\tau)  )\ d\tau,
$$
By integrating function $l$ along $\gamma_{x,t} $ with $\gamma_{x,t}(0)=x$ over the interval $[-\sigma , 0]$, we obtain

\begin{align*}
	&\, \int^0_{-\sigma} e^{\lambda \tau }  l ( \gamma_{x,t}(\tau),\dot \gamma_{x,t}(\tau)  )\ d \tau +  u^-(x)- e^{-\lambda \sigma  }u^-(\gamma_{x,t}(-\sigma )  )\\
	=&\, \int^0_{-\sigma} e^{\lambda \tau } L( \gamma_{x,t}(\tau),\dot \gamma_{x,t}(\tau)  )\ d\tau =  U_{\psi}(x )- e^{-\lambda \sigma  } U_{\psi}(\gamma_{x,t }(-\sigma),-\sigma) 
\end{align*}
which implies that
\begin{equation}\label{eq:lu}
	\begin{split}
		 &\,\int^0_{-\sigma} e^{\lambda \tau }  l ( \gamma_{x,t}(\tau),\dot \gamma_{x,t}(\tau)  )\ d \tau  \\
		\leqslant &\, e^{-\lambda \sigma} \Big| u^-(\gamma_{x,t}(-\sigma )) - U_{\psi}(\gamma_{x,t}(-\sigma),-\sigma ) \Big|  +  |u^-(x) -  U_{\psi}(x,t) |, \\
	\leqslant &\, e^{-\lambda \sigma} C_1(\psi,L , \lambda)e^{-\lambda (t-\sigma)}+   C_1(\psi,L, \lambda ) e^{- \lambda t} \\
	=&\, 2 C_1(\psi,L, \lambda ) e^{ -\lambda t}
	\end{split}
\end{equation}
where $C_1(\psi, L)$ is the same constant as in item (6) of Proposition \ref{prop-2}. On the other hand, we denote 
\beq\tag{$\star\star$}
F(x,w):=l(x,w)+H(x,d_x u^-(x))+\lambda u^-(x) ,\quad (x,w)\in T\T^n
\eeq
which verifies to be nonnegative by the Fenchel Transform. Moreover,
\begin{align*}
	\frac{\partial F}{\partial w}(x,w)=&\, \frac{\partial l}{\partial w}(x,w)= \frac{\partial L}{\partial w}(x,w)- d_x u^-(x) \\
  \frac{\partial ^2 F}{\partial  w^2 } ( x,w)=&\, \frac{\partial ^2 l}{\partial  w^2 } (x,w) = \frac{\partial ^2 L}{\partial  w^2 } (x,w)
\end{align*}
then $F(x,w)=0$ if and only if $w=\frac{\partial H}{\partial p}(x,d_xu^-(x) )$. Due to (H1), there exists $\alpha_0(\psi,L,\lambda)$ such that
\be\label{eq:F convex}
	F(\gamma_{x,t}(\tau),\dot \gamma_{x,t}(\tau) )\geqslant \alpha_0(\psi,L,\lambda)\Big|\dot \gamma_{x,t}(\tau)-\frac{\partial H}{\partial p}(\gamma_{x,t}(\tau) , d_x u^-(\gamma_{x,t}(\tau) )) \Big|^2.
\ee
Recall that $H(x,d_x u^-(x),t)+\lambda u^-(x)=0 $ and we introduce 
\[
\eta_{x,t}(r ):=x-\int^0_r \frac{\partial H}{\partial p}(\gamma_{x,t}(\tau) , d_x u^-(\gamma_{x,t}(\tau) )) \ d \tau 
\]
 for any $ - \sigma < r< 0 $. Thus
\beq\label{eq:substitute-curve}
\dot\eta_{x,t} (\tau) = \frac{\partial H}{\partial p}(\gamma_{x,t}(\tau) , d_x u^-(\gamma_{x,t}(\tau) )) .
\eeq
Now from \eqref{eq:F convex}, we get
\begin{align*}
	 l ( \gamma_{x,t}(\tau),\dot \gamma_{x,t}(\tau)  )\geqslant &\,\alpha_0(\psi,L,\lambda)\Big|\dot \gamma_{x,t}(\tau)-\frac{\partial H}{\partial p}(\gamma_{x,t}(\tau) , d_x u^-(\gamma_{x,t}(\tau) )) \Big|^2\\
	 =&\, \alpha_0(\psi,L,\lambda)\big|\dot \gamma_{x,t}(\tau)- \dot \eta_{x,t} (\tau)\big|^2
\end{align*}
which completes the proof.
\end{proof}


Recall that $\cT_\om$ is the unique attractor in its local neighborhood $\cU$ (see Appendix \ref{a2}). There exists a constant $\Delta_0=\Dt_0(\lb,L,\psi)>0$ such that  
\[
 dist(\cT_\om,\partial \cU ) \geqslant \frac{3}{2}\Delta_0  C_2(\psi, L, \lambda) 
 \]
  can be guaranteed by chosing suitable $\cU$. Consequently, the following Lemma holds:

\begin{lem}\label{lemma:speed distance}
	There exists a suitable constant $s_0=s_0(\psi,L,\lb)\geqslant 1$ such that   
$$
\lim_{\tau \to 0_- } | \dot \gamma_{x,t}(\tau )-\dot \gamma_x^-(\tau ) | \leqslant   \frac{3}{2} \Delta_0, \quad\quad  \forall \ t \in [s_0,+\infty) , \; x\in \T^n
$$ 
\end{lem}
\begin{proof}


For any $s\geqslant 1$, 
we assume that 
\beq\label{eq:assume}\tag{$\odot$}
\lim_{\tau \to 0_- } | \dot \gamma_{x,t}(\tau )-\dot \gamma_x^-(\tau ) | > \Delta_0
\eeq
 for some $t\geqslant s$ and $x\in\T^n$. Otherwise, the assertion of this Lemma holds. 
 Due to \eqref{compact}, there exists a constant 
\[
A_0(\psi,L):=  \max\bigg\{A, \max_{(x,v) \in \T^n\times B(0,A)  }\frac{\partial H}{\partial p}(x,L_v(x,v))\bigg\}
\]
 such that for any $\tau \in [-t,0]$
\begin{align*}
	| \dot \eta_{x,t}(\tau) |=&\,\Big|\frac{\partial H}{\partial p}(\gamma_{x,t}(\tau) , d_x u^-(\gamma_{x,t}(\tau) )) \Big|= \Big|\frac{\partial H}{\partial p}(\gamma_{x,t}(\tau) , L_v(\gamma_{x,t}(\tau), \dot \gamma_{x,t}(\tau)) \Big| \leqslant A_0
\end{align*}
where $B(0,A)$ is a disk centering at $0$ and of a radius $A:=A(\psi,L, \lambda)$ which has been given in item (3) of Lemma \ref{lem: converange prop}. On the other side,  there exists $\sigma \in [0,t]$ such that 
\[
 | \dot \gamma_{x,t}( \tau ) - \dot\eta_{x,t} (\tau)   | > \Delta_0 ,\quad \tau \in (-\sigma,0]
 \]
 and $ | \dot \gamma_{x,t}( -\sigma ) - \dot\eta_{x,t} (-\sigma)   |=\Dt_0$ as long as $s$ suitably large. This is because 
\begin{align*}
	2C_1(\psi,L, \lambda) e^{-\lambda t} \geqslant	 &\,\int^0_{-\sigma} e^{\lambda \tau }  l ( \gamma_{x,t}(\tau),\dot \gamma_{x,t}(\tau)  )\ d \tau \\
	\geqslant &\, \alpha_0 \int^0_{-\sigma} e^{\lambda \tau }  \big|\dot \gamma_{x,t}(\tau)- \dot\eta_{x,t} (\tau)\big|^2\ d \tau \\
	\geqslant &\, \frac{\alpha_0}{\lambda} \min_{\tau\in [-\sigma,0] }\big|\dot \gamma_{x,t}(\tau)- \dot\eta_{x,t} (\tau)\big|^2 (1-e^{-\lambda \sigma })\\
	\geqslant &\, \frac{\alpha_0}{\lambda} \Delta_0^2 (1-e^{-\lambda \sigma }).
\end{align*}
due to \eqref{eq:estimate l}. Therefore,  if we make 
\[
s_0:= \max\bigg\{-\frac{1}{\lambda}\ln \frac{\alpha_0 \Delta_0^2}{\lambda C_1(\psi,L,\lambda)} ,\;  -\frac{1}{\lambda} \ln \frac{\alpha_0 \Delta_0^2}{C_1(\psi,L, \lambda)} \big ( 1-e^{-\frac{\lambda \Delta_0}{4A_0}}  \big),\; 1\bigg\}
\]
 then 
 $
 \sigma\leqslant \frac{\Delta_0}{4A_0}
 $
 since $t\geqslant s\geqslant s_0$. Due to \eqref{eq:substitute-curve},
\begin{align*}
	\lim_{\tau \to 0_- } | \dot \gamma_{x,t}(\tau )-\dot \gamma_x^-(\tau ) | =&\, \lim_{\tau \to 0_- } | \dot \gamma_{x,t}(\tau )-\dot \eta_{x,t}^-(\tau ) |\\
 \leqslant  &\, \max_{\tau \in(-\sigma ,0)}  | \dot \gamma_{x,t}(\tau )-\dot \eta_{x,t}^-(\tau ) |  \\
  \leqslant &\,      \big| \dot \gamma_{x,t}(-\sigma )- \dot\eta_{x,t} (-\sigma )\big|+ \sigma\cdot \max_{\tau\in(-\sigma,0)}\big|\ddot \gamma_{x,t}(\tau )- \ddot\eta_{x,t} (\tau ) \big| \\
	 \leqslant  &\Dt_0+2A_0 \sigma  \leqslant\, \frac{3}{2}\Delta_0 
\end{align*}
 for any possible $t\geqslant s $ and $x\in\T^n$ such that \eqref{eq:assume} holds. So we  complete the proof.
\end{proof}


\vspace{20pt}

\noindent{\it Proof of Theorem \ref{thm:2}:}
For any $x \in \T^n $ and  $ t \geqslant s_0(\psi,L,\lambda ) $ with $s_0 $ given in Lemma \ref{lemma:speed distance}, there holds
\ben
& & \max_{x\in\T^n}d_H \bigg( \Pi_x D^+\cT_{t}^-\psi(x), 
      P(x)   \bigg)  \\
& = &\max_{x\in\T^n}d_H(\Pi_x D^*U_{\psi}(x,t) ,d_xu^- )  \\
&= &\max_{x\in\T^n}\max_{p_x\in \Pi_x D^*U_{\psi}(x,t) }  \big| p_x-d_xu^-  \big| 
\een
due to Lemma \ref{lem: converange prop}, where $d_H(\cdot,\cdot):2^{\R^n}\times 2^{\R^n}\rightarrow\R$ is the {\sf Hausdorff distance} between any two subsets of $\R^n$. Besides, due to Lemma \ref{lemma:speed distance}, for any $p_x\in \Pi_x D^*U_{\psi}(x,t)$ and  $t\geqslant s_0$, it obtains that 
$$
|p_x-d_xu^-|\leqslant  C_2(\psi,L,\lambda)  \lim_{\tau \to 0_- } | \dot \gamma_{x,t}(\tau )-\dot \gamma_x^-(\tau ) |\leqslant \frac{3}{2} C_2(\psi,L,\lambda) \Delta_0
$$
which implies that $(x,p_x)$ lies in the neighborhood $\cU$ of KAM tours $\cT_\omega $, i.e.
\be\label{eq:s0 in U}
(x,p_x)\in \cU  \quad \forall \ t\geqslant s_0.
\ee

Next, we claim that there exists  a constant $C_3(\psi,L, \lambda )>0$  depending only on $\psi, L$ and $\lambda $, such that 
\be\label{claim lambda}
|p_x-d_xu^-|\leqslant C_3(\psi,L,\lambda )  e^{-\lambda t}, \quad \forall \ t\geqslant s_0.
\ee
Due to \eqref{eq:s0 in U}, 
$$
t_0:=\sup\{ s\geqslant 0| \Phi_{H,\lambda}^{-\tau} (x,p_x) \in \cU,\forall \tau\in[0,s] \}
$$
is always finite. Since  $\cT_\omega=\{y,P(y)| y\in \T^n \}$ is a Lipschitz graph with Lipschitz constant $l:=l(\psi,L,\lambda)$ , then  
$$\cT_\omega \subset S(x,l):= \big\{(y,z) \in T^*\T^n \big|  |z-P(x)| \leqslant l|y-x| \big\}$$
 which $S(x,l)$ is a cone clustered at $(x,P(x))\in T^*\T^n$. Therefore, 
 \be\label{eq:attactor2}
\begin{split}
	|p_x-d_x u^-|= &\, |(x,p_x)-(x,P(x))|\geqslant dist(  (x,p_x) ,\cT_\omega )\\
	 \geqslant &\, dist(  (x,p_x),S(x,l))=\frac{1}{\sqrt{l^2+1}} |p_x-d_x u^- |.
\end{split}
\ee
On the other side, Proposition \ref{prop:local-stab} implies $\cT_\om$ is normally hyperbolic with the Lyapunov exponent $-\lb<0$. There exists constants $C_4(\psi,L,\lambda ),$ $C_5(\psi,L,\lambda )>0$  depending only on $\psi,L$ and $\lambda$,  such that 
\be\label{eq:acctor1}
\begin{split}
	C_4(\psi,L,\lambda ) e^{\lambda s}dist(  (x,p_x) ,\cT_\omega ) \leqslant &\,dist( \Phi_{H,\lambda}^{-s} (x,p_x) ,\cT_\omega ) \\ \quad \leqslant &\, dist(  (x,p_x) ,\cT_\omega )C_5(\psi,L,\lambda) e^{\lambda s},  \quad \forall  s \in [0,t_0] .
\end{split}
\ee
Benefiting from \eqref{eq:acctor1}, we conclude
$$ 
 e^{\lambda t_0}C_5(\psi,L,\lambda)|p_x-d_xu^-|\geqslant   dist( \Phi_{H,\lb}^{-t_0} (x,p_x) ,\cT_\omega )  \geqslant C_0:= \frac{1}{4}\text{diam} \ \cU. 
$$
Consequently,
\be\label{eq:t0}
t_0 \geqslant  \frac{C_0}{C_5(\psi,L,\lambda) }  \cdot (-\frac{1}{\lambda} \ln |p_x-d_xu^-| ) . 
\ee
For any $\tau\in(-t_0,0) $ and $\eta_{x,t}:[-t,0]\to \T^n$ defined as in \eqref{eq:curve eta}, we have the following estimate:
\be\label{eq:estimate 2 curve speed}
\begin{split}
	 \big|\dot \gamma_{x,t}(\tau)- \dot \eta_{x,t} (\tau)\big|^2 
	=  &\,   \Big|\dot \gamma_{x,t}(\tau)-\frac{\partial H}{\partial p}(\gamma_{x,t}(\tau) , d_x u^-(\gamma_{x,t}(\tau) )) \Big|^2\\
	 =&\,  \Big|\frac{\partial H}{\partial p}(\gamma_{x,t}(\tau),d_xU_\psi(\gamma_{x,t}(\tau),\tau ))-\frac{\partial H}{\partial p}(\gamma_{x,t}(\tau) , d_x u^-(\gamma_{x,t}(\tau) )) \Big|^2 \\
	 =&\, \Big|Q(\tau)\cdot\Big(d_xU_\psi(\gamma_{x,t}(\tau),\tau ) - d_x u^-(\gamma_{x,t}(\tau) )\Big)\Big|^2 \\
	 \geqslant &\, \beta^2  \Big|dist( \Phi_{H,\lambda}^{\tau} (x,p_x) ,\cT_\omega )\Big|^2 \\
	 \geqslant &\, \beta^2 C_4^2(\psi,L,\lambda) e^{-2\lambda \tau}  dist^2( (x,p_x) ,\cT_\omega ) \\
	 \geqslant &\, \frac{\beta^2   C_4^2(\psi,L,\lambda) e^{-2\lambda \tau}}{l^2+1}   |p_x-d_xu^-|^2 
\end{split}
\ee
with
\[
Q(\tau):=\int_0^1 \frac{\partial^2 H}{\partial p^2}\Big(\gamma_{x,t}(\tau) ,\sigma d_xU_\psi(\gamma_{x,t}(\tau),\tau )+(1-\sigma)d_x u^-(\gamma_{x,t}(\tau) )  \Big) d\sigma. 
\]
Due to (H1) and item (2) of Proposition \ref{prop-1}, $Q(\tau)$ is uniformly positive definite 
for $\tau\in(-t_0,0)$, namely $Q(\tau)\geqslant \beta Id_{n\times n}$ for some constant $\beta>0$. The second and third inequality of \eqref{eq:estimate 2 curve speed} is due to 
\eqref{eq:acctor1} and \eqref{eq:attactor2} respectively. Taking \eqref{eq:estimate 2 curve speed} into \eqref{eq:estimate l} we get 
\be\label{eq:act est}
\begin{split}
	 2  C_1(\psi,L ) e^{- \lambda t} \geqslant&\, \int^{0}_{- t} e^{\lambda \tau }  l ( \gamma_{x,t}(\tau),\dot \gamma_{x,t}(\tau))\ d \tau \\
	 \geqslant &\, \alpha_0 \int^{0}_{-t_0} e^{\lambda \tau }  \big|\dot \gamma_{x,t}(\tau)- \dot \eta_{x,t} (\tau)\big|^2\ d \tau  \\
	 \geqslant &\, \frac{\alpha_0  \beta^2   C_4^2(\psi,L,\lambda) }{l^2+1} |p_x-d_x u^-(x)|^2 \int_{-t_0}^0   e^{-\lambda \tau}  d\tau \\
	 =&\,\frac{\alpha_0 \beta^2   C_4^2(\psi,L,\lambda) }{\lambda l^2+\lambda}  (e^{\lambda t_0} -1) |p_x-d_x u^-(x)|^2.
\end{split}
\ee
Taking account of  \eqref{eq:t0}, previous \eqref{eq:act est} implies the existence of a constant $C_3(\psi,L,\lb)>0$ (depending only on $\psi$, $L$ and $\lb$) such that  
$$
|p_x-d_xu^-|\leqslant C_3(\psi,L,\lb)  e^{-\lambda t}, \quad \forall \ t\geqslant s_0.
$$
so our claim \eqref{claim lambda} get proved. Finally, 
\ben
\|\cT_t^-\psi(x)-u^-_\om(x)\|_{W^{1,\infty}}&=&\max_{x\in\T^n}d_H \bigg( \Pi_x D^+\cT_{t}^-\psi(x), P(x)   \bigg) +\|\cT_t^-\psi(x)-u^-_\om(x)\|_{L^{\infty}}\\
      &=&\max_{x\in\T^n}\max_{p_x\in \Pi_x D^*U_{\psi}(x,t) }  \big| p_x-d_xu^-  \big| +C_1(\psi, L)e^{-\lb t}\\
      &\leqslant& (C_3+C_1) e^{-\lambda t}.
\een
where $C_1(\psi,L,\lambda)$ is given by item (6) of Proposition \ref{prop-2}. By taking $C=C_3+C_1$ we get the main conclusion of Theorem \ref{thm:2}.\qed

\medskip

\section{Persistence of KAM torus as invariant Lagrangian graph}\label{a1}

\vspace{10pt}

This section is devoted to prove Theorem \ref{thm:3}. Firstly, the following notions and conclusions are needed:

\begin{defn}[Aubry Set]\label{defn:aubry}
$\gm\in C^{ac}(\R,\T^n)$ is called {{\sf globally calibrated}} by $u^-$, if for any $a\leqslant b\in\R$,
\[
e^{\lb b}u^-(\gm(b))-e^{\lb a} u^-(\gm(a))=\int_a^be^{\lb t}L(\gamma(t),\dot\gamma(t))dt.
\]
The {\sf Aubry set} $\wt \cA$ is an $\Phi_{L,\lb}^t-$invariant set defined by
\[
\wt\cA=\bigcup_{\gamma}\,\,\{(\gamma, \dot\gamma)|\gamma \text{ is globally calibrated by }u^-\}\subset T\T^n
\]
and the {\sf projected Aubry set} can be defined by $\cA=\pi\wt\cA\subset\T^n$, where $\pi:(x,p)\in T\T^n\rightarrow x\in\T^n$ is the standard projection.
\end{defn}
\begin{prop}\cite{MS}
 $\pi^{-1}:\cA\rightarrow TM$ is a Lipschitz graph.
\end{prop}

\begin{prop}[Upper semicontinuity]\label{lem:u-semi}
As a set-valued function defined on $ C^{r\geqslant 2}(T\T^n,\R)$,
\[
\wt\cA:\big\{C^{r\geqslant 2}(T\T^n,\R),\|\cdot\|_{C^r}\big\}\longrightarrow \big\{T\T^n,d_H\big\}
\]
is upper semicontinuous.
\end{prop}
\begin{proof}
It suffices to prove that for any $L_n$ accumulating to $L$ w.r.t. the $\|\cdot\|_{C^r}-$norm as $n\rightarrow +\infty$, any accumulating curve of 
$\{\gamma_n\in\cA(L_n)\}$ would be contained in  $\cA(L)$.

Due to item (5) of Proposition \ref{prop-2}, for any $L_n$ converging to $L$ w.r.t. the $\|\cdot\|_{C^r}-$norm, $\wt\cA(L_n)$ is uniformly compact in the phase space. Therefore, for any sequence $\{\gamma_n\}$ globally calibrated by $L_n$, any accumulating curve $\gamma_*$ has to satisfy
\[
\int_t^se^{\lb\tau}L(\gamma_*,\dot\gamma_*)d\tau= \lim_{n\rightarrow+\infty}\int_t^se^{\lb\tau}L(\gamma_n,\dot\gamma_n)d\tau=\lim_{n\rightarrow+\infty}\int_t^se^{\lb\tau}L_n(\gamma_n,\dot\gamma_n)d\tau
\]
for any $t<s\in\R$. On the other side, for any $\eta\in C^{ac}([t,s],M)$ satisfying $\gamma_*(s)=\eta(s)$ and $\gamma_*(t)=\eta(t)$, we can find $\eta_m\in C^{ac}([t_n,s_n],M)$ ending with $\gamma_n(s_n)=\eta_n(s_n)$ and $\gamma_n(t_n)=\eta_n(t_n)$ for some sequence $\{\gamma_n\in\cA(L_n)\}_{n\in\N}$ with
$[t_n,s_n]\subset [t,s]$ for all $n\in\N$, 
\[
\lim_{n\rightarrow+\infty}t_n=t \ \text{(resp. $\lim_{n\rightarrow+\infty}s_n=s$)}
\]
 and $\eta_n\rightarrow \eta$ uniformly on $[t,s]$ as $m\rightarrow+\infty$. Since $\gamma_n\in\cA(L_n)$ for all $n\in\N$, then 
\[
\int_{t_n}^{s_n}e^{\lb\tau}L_n(\eta_n,\dot\eta_n)d\tau\geqslant \int_{t_n}^{s_n}e^{\lb\tau}L_n(\gamma_n,\dot\gamma_n)d\tau.
\]
Combining these conclusions we get 
\ben
\int_t^se^{\lb\tau}L(\gamma_*,\dot\gamma_*)d\tau&=& \lim_{n\rightarrow+\infty}\int_t^se^{\lb\tau}L(\gamma_n,\dot\gamma_n)d\tau\\
&= &\lim_{n\rightarrow+\infty}\int_{t_n}^{s_n}e^{\lb\tau}L_n(\gamma_n,\dot\gamma_n)d\tau\\
&\leqslant & \lim_{n\rightarrow+\infty}\int_{t_n}^{s_n}e^{\lb\tau}L_n(\eta_n,\dot\eta_n)d\tau\\
&=&\int_t^se^{\lb\tau}L(\eta,\dot\eta)d\tau,
\een
which indicates $\gamma_*:\R\rightarrow M$ minimizes $h_\lb^{s-t}(\gamma_*(t),\gamma_*(s))$ for any $t<s\in\R$. So  $\gamma_*\in\cA(\lb_*)$.
\end{proof}

\bigskip

\noindent{\it Proof of Theorem \ref{thm:3}:}
Suppose $\cT$ is a KAM torus of \eqref{eq:ode0} associated with $H(x,p)$ (no constraint on the frequency).
Due to the Invariant {Manifold Theorem} \cite{HPS}, for system 
\[
H_\eps(x,p)=H(x,p)+\eps H_1(x,p),\quad0<\eps\leqslant \eps_0(H)\ll1,
\]
 the perturbed invariant graph $\cT^\eps$ is still normally hyperbolic, as long as the constant $\eps_0(H)$ is sufficiently small. 
 

Recall that $\cT=\cL_H^{-1}(\wt\cA(H))$ (see Definition \ref{defn:aubry} for $\wt\cA$). Due to the upper semi-continuity of the Aubry set $\wt\cA$ (see Lemma \ref{lem:u-semi}), the Hausdorff distance between $\wt\cA(H)$ and $\wt\cA(H_\eps)$ can be sufficiently small as long as $\eps_0\ll1$, which implies $\wt\cA(H_\eps)\subset\cL_{H_\eps}(\cT^\eps)$ (but may not be equal).

Suppose $u^-_\eps(x)$ is the unique viscosity solution of (\ref{eq:sta-hj}) for system $H_\eps$, then it's differentiable a.e. $x\in\T^n$. For any differentiable point $x\in\T^n$,  $(x,d_x u_\eps^-)$ decides a unique backward orbits which tends to $\wt\cA^\eps$ as $t\rightarrow-\infty$ (shown in Proposition \ref{prop-2}), so $(x,d_x u_\eps^-)\in\cT^\eps$. Then $Graph(d_xu_\eps^-)$ coincides with $\cT^\eps$ for a.e. $x\in\T^n$. Since $\cT^\eps$ is graphic and at least $C^1-$smooth, then $u_\eps^-$ is actually a classic solution of (\ref{eq:sta-hj}) for system $H_\eps$ and then $Graph(d_xu_\eps^-)=\cT^\eps$, which implies $\cT^\eps$ is exact, therefore has to be  Lagrangian.\qed


\vspace{20pt}

\appendix

\section{Normal hyperbolicity of the KAM torus for CSTMs}\label{a2}
\vspace{10pt}

In this Appendix, we show the normal hyperbolicity of the KAM torus for conformally symplectic system (\ref{eq:ode0}). This conclusion was firstly proved in \cite{CCD1,CCD}. 
Nonetheless, we reprove it here for the consistency.

\begin{prop}[local attractor\cite{CCD1,CCD}]\label{prop:local-stab}
The KAM torus $\cT_\om$ is a normally hyperbolic $\Phi_{H,\lambda}^t$-invariant manifold, consequently, there exists a suitable neighborhood $\cU$ of it such that $\cT_\om$ is the $\om-$limit set of any point $x\in\cU$.
\end{prop}
\begin{proof}
Since the KAM torus $\cT_\om$ is $\Phi_{H,\lambda}^t$-invariant, so we just need to prove its normal hyperbolicity w.r.t. $\Phi_{H,\lb}^1$. The generalization from $\Phi_{H,\lb}^1$ to $\Phi_{H,\lb}^t$ is straightforward. Due to (\ref{diam-com}), we know $\Phi_{H,\lb}^1\circ K(\theta)=K(\theta+\om)$, which implies
\[
D\Phi_{H,\lb}^1(K(\theta))\partial_{i} K(\theta)=\partial_i K(\theta+\om),\quad\forall \theta\in\T^n,\ i=1,2,\cdots,n .
\]
Therefore, $\partial_i K(\theta)\in T_{K(\theta)}T^*\T^n$ is an eigenvector of $D\Phi_{H,\lb}^1(\cdot)$ of the eigenvalue $1$. As $\{K(\theta)|\theta\in\T^n\}$ is a Lagrangian graph, i.e.
\[
\Om( \partial_i K(\theta),\partial_j K(\theta))=0,\quad\forall \theta\in\T^n,\  i,j=1,\cdots,n,
\]
so we have
\[
span_{i=1,\cdots,n}\{\partial_i K(\theta)\oplus J\partial_i K(\theta)\}=T_{K(\theta)}T^*\T^n.
\]
Formally for the matrix
\[
V(\theta)=J_{2n\times 2n}^tDK(\theta)\Big(DK^{t}(\theta)\cdot DK(\theta)\Big)^{-1}
\]
with $DK(\theta)=(\partial_1 K(\theta),\cdots,\partial_n K(\theta))$, we have
\ben
D\Phi_{H,\lb}^1(K(\theta))V(\theta)=V(\theta+\om)\cdot A(\theta)+DK(\theta+\om)\cdot S(\theta)
\een
where
\[
A(\theta)=e^{-\lb}Id
\]
and
\[
S(\theta)=DK^{-1}(\theta+\om)\big[D\Phi_{H,\lb}^1(K(\theta))V(\theta)-e^{-\lb}V(\theta+\om)\big].
\]
This is because the conformally symplectic condition implies 
\[
\Big(D\Phi_{H,\lb}^1(K(\theta))\Big)^t J D\Phi_{H,\lb}^1(K(\theta))=e^{-\lb} J.
\]
Defining $M(\theta)$ by a $2n\times 2n$ matrix which is juxtaposed with $DK(\theta)$ and $V(\theta)$, i.e.
\[
M(\theta)=\bigg(DK(\theta)\Big|V(\theta)\bigg),
\]
we will see
\[
D\Phi_{H,\lb}^1(K(\theta))M(\theta)=M(\theta+\om)\begin{pmatrix}
  Id   &   S(\theta) \\
   0   &  e^{-\lb}Id
\end{pmatrix}_{2n\times 2n}.
\]
Recall that
\[
E^c:=\bigg\{\Big(K(\theta),D K(\theta)\Big)\bigg|\theta\in\T^n\bigg\}
\]
 is a $\Phi_{H,\lb}^1-$invariant subbundle of $T_{\cT_\om}T^*\T^n$, with the eigenvalue $1$. To find the other $\Phi_{H,\lb}^1-$ invariant subbundle, we assume there exists a $n\times n$ matrix $B(\theta)$ such that
\[
E^s:=\bigg\{\Big(K(\theta),D K(\theta)B(\theta)+V(\theta)\Big)\bigg|\theta\in\T^n\bigg\}
\]
is $\Phi_{H,\lb}^1-$invariant, then
\ben
D\Phi_{H,\lb}^1(K(\theta))(E^c|E^s)&=&D\Phi_{H,\lb}^1(K(\theta))M(\theta)\begin{pmatrix}
  Id    &  B(\theta)  \\
    0  &  Id
\end{pmatrix}\\
&=&M(\theta+\om)\begin{pmatrix}
  Id   &   S(\theta) \\
   0   &  e^{-\lb}Id
\end{pmatrix}\begin{pmatrix}
  Id    &  B(\theta)  \\
    0  &  Id
\end{pmatrix}\\
&=&\Big(D K(\theta+\om)\Big|D K(\theta+\om)B(\theta+\om)+V(\theta+\om)\Big)\begin{pmatrix}
  Id    &  -B(\theta+\om)  \\
    0  &  Id
\end{pmatrix}\cdot\\
& &\begin{pmatrix}
  Id   &   S(\theta) \\
   0   &  e^{-\lb}Id
\end{pmatrix}\begin{pmatrix}
  Id    &  B(\theta)  \\
    0  &  Id
\end{pmatrix}\\
&=&\Big(D K(\theta+\om)\Big|D K(\theta+\om)B(\theta+\om)+V(\theta+\om)\Big)U(\theta+\om)
\een
where
\[
U(\theta+\om)=\begin{pmatrix}
     Id &  -e^{-\lb}B(\theta+\om)+S(\theta)+ B(\theta) \\
   0   &  e^{-\lb}Id
\end{pmatrix}
\]
has to be diagonal. That imposes
\[
-e^{-\lb}B(\theta+\om)+S(\theta)+ B(\theta) =0,\quad\forall \theta\in\T^n.
\]
We can always find a suitable $B(\theta)$ solving this equation, since
%
 there is no small divisor problem and the regularity of $B(\cdot)$ keeps the same with $S(\cdot)$. So $E^s$ is indeed a $\Phi_{H,\lb}^1-$invariant subbundle with the eigenvalue $e^{-\lb}<1$.

Now we get an invariant splitting of $T_{\cT_\om}T^*\T^n$ by $E^c\oplus E^s$, with the eigenvalue $1$ and $e^{-\lb}$ respectively. Due to the Invariant Manifold Theorem \cite{F1,F2}, we can prove the normal hyperbolicty of $\cT_\om$ (which is actually normally compressible in the forward time). So $\cT_\om$ is a local attractor.
\end{proof}

\vspace{40pt}
\bibliographystyle{plainurl}
\bibliography{attractionKAM}
\end{document}